\documentclass[11pt]{article}
\title{Weakly Semi-Preopen and Semi-Preclosed Functions in $L$-fuzzy Topological Spaces}
\author{  A. Ghareeb\smallskip\\
{\it Department of Mathematics, Faculty of Science, South
Valley University, Qena,  Egypt.}\\ {E-mail: nasserfuzt@aim.com.}}
\usepackage{time}
\usepackage{amsthm}
\usepackage{color,xypic}
\usepackage{hyperref}
\theoremstyle{definition}
\newtheorem{defn}{Definition}[section]
\newtheorem{thm}{Theorem}[section]

\newtheorem{cexa}{Counter Example}[section]
\theoremstyle{remark}
\newtheorem*{rmk}{Remark}
\newcommand{\To}{\longrightarrow}

\begin{document}
\date{}
\maketitle
\begin{abstract}
A new class of functions called $L$-fuzzy weakly Semi-Preopen (Semi-Preclosed) functions in $L$-fuzzy topological spaces are introduced in this paper. Some characterizations of this class and its properties and the relationship with other classes of functions between $L$-fuzzy topological spaces are also obtained.\medskip

\noindent\textbf{keywords:} $L$-fuzzy topology; $L$-fuzzy weakly Semi-Preopen function; $L$-fuzzy weakly Semi-Preclosed function.\medskip

\end{abstract}

\section{Introduction}

The fuzzy concept has invaded almost all branches of Mathematics since the introduction
of the concept by Zadeh \cite{zadeh}. Fuzzy sets have applications in many fields
such as information \cite{Smets} and control \cite{Sugeno}. The theory of fuzzy topological spaces was introduced and developed by Chang \cite{Chang} and since then various notions in classical topology have been extended to fuzzy topological spaces. \v{S}ostak \cite{sostak} and Kubiak \cite{Kubiak} introduced the fuzzy topology as an extension of Chang's fuzzy topology. It has been developed in many directions. \v{S}ostak \cite{sostak1} also published a survey article of the developed areas of fuzzy topological spaces. The topologistes used to call Chang's fuzzy topology by "$L$-topology" and Kubiak-\v{S}ostak's fuzzy topology by "$L$-fuzzy topology" where $L$ is an appropriate lattice.\medskip

The concept of $r$-fuzzy Preopen and $r$-fuzzy Preclosed subsets was introduced in \cite{kim1}. In \cite{kim1} this concept has been used to introduce new kinds of functions in $L$-fuzzy topological spaces (in case $L=[0,1]$). Our motivation in this paper is to define new functions in $L$-fuzzy topological spaces based on the concepts $r$-fuzzy Semi-Preopen and $r$-fuzzy Semi-Preclosed subsets which has been defined in \cite{kim1}  and investigate their properties.

\section{Preliminaries}

Throughout this paper $(L,\le ,\bigwedge ,\bigvee ,')$  is a
complete DeMorgan algebra, $X$ is a nonempty set. $L^{X} $ is the
set of all $L$-fuzzy subsets  on $X$. The smallest element and the largest element in $L^{X} $ are denoted by $\underline \bot$ and $\underline \top$, respectively.  A complete lattice $L$ is a complete Heyting algebra if it satisfies the following infinite distributive law: For all $a\in L$ and all $B\subset L$, \[a\wedge\bigvee B=\bigvee\{a\wedge b\,|\,b\in B\}.\]

An element $a$ in  $L$ is called a prime element if $a\ge b\wedge c$
implies $a\ge b$ or $a\ge c$. An element $a$ in $L$ is called
co-prime if $a'$ is prime \cite{6}. The set of non-unit prime
elements in  $L$ is denoted by $P(L)$. The set of non-zero co-prime
elements in $L$ is denoted by $J(L)$.\smallskip

The binary relation $\prec $ in  $L$ is defined as follows: for $a$,
$b\in L$,  $a\prec b$ if and only if for every subset  $D\subseteq
L$, the relation $b\le sup \,D$ always implies the existence of
$d\in D$ with  $a\le d$ \cite{7}. In a completely distributive DeMorgan algebra $L$, each element  $b$ is a sup of $\{ a\in L|a\prec
b\} $. A set $\{ a\in L|a\prec b\}$ is called the greatest minimal
family of  $b$ in the sense of \cite{9,10}, denoted by $\beta (b)$,
and $\beta ^{*} (b)=\beta (b)\cap J(L)$.  Moreover, for $b\in L$, we
define $\alpha (b)=\{ a\in L|a'\prec b'\}$ and $\alpha ^{*}
(b)=\alpha (b)\cap P(L)$.\smallskip

An L-fuzzy point in $L^X$  is an L-fuzzy subset $x_{\lambda}$, where $\lambda\in L_{\bot}=L-\{\bot\}$, such that $x_{\lambda}(y)=\lambda$ when $y=x$ and $\bot$ otherwise. For L-fuzzy subsets $U$, $V\in L^X$, we write $U q V$ to mean that $U$
is quasi-coincident (q-coincident, for short) with $U$, i.e., there
exists at least one point $x\in X$ such  that $U(x)\not\leq U(x)'$. Negation of such a statement is denoted as $U \neg q V$. Let $f:X\rightarrow Y$ be a crisp mapping.
Then an L-fuzzy mapping $f_L^{\rightarrow}:L^{X}\rightarrow L^{Y}$ is
induced by $f$ as usual, i.e., $f_L^{\rightarrow}(U)(y)=\bigvee_{x\in X,\quad f(x)=y}U(x)$  and  $f_L^{\leftarrow}(V)(x)=V(f(x))$.\medskip

An $L$-topological space (or $L$-space, for short) is a pair $(X, \tau)$, where  $\tau$ is a subfamily of $L^{X} $ which contains $\underline \bot$;  $\underline \top$ and is closed for any suprema and finite infima. $\tau$ is called an $L$-topology on  $X$. Members of $\tau$ are called open $L$-fuzzy subsets and their complements are called closed $L$-fuzzy subsets.\medskip

\begin{defn}[\cite{Kubiak,sostak}]A function $\mathcal{T}\colon L^{X}\To L$ is called an  L-fuzzy
topology on $X$ if it satisfies the following conditions:
\begin{enumerate}
\item[(O1)]$\mathcal{T}(\underline{\bot})=\mathcal{T}(\underline{\top})=\top$.
\item[(O2)]$\mathcal{T}(U \wedge V)\geq \mathcal{T}(U)\wedge\mathcal{T}(V)$ for each $U$, $V\in L^{X}$.
\item[(O3)]$\mathcal{T}(\bigvee_{i\in\Gamma}U_{i})\geq \bigwedge_{i\in \Gamma}\mathcal{T}(U_{i})$ for any $\{U_{i}\}_{i\in\Gamma}\subset L^{X}$.
\end{enumerate}
\end{defn}
The pair $(X,\mathcal{T})$ is called an L-fuzzy topological spaces. $\mathcal{T}(U)$ can be interpreted as the degree to which $U$ is an open $L$-fuzzy subset and $\mathcal{T}(U')$ will be called the degree of closedness of $U$, where $U'$ is the $L$-complement of $U$.  A function $f:(X,\mathcal{T}_1)\rightarrow (Y,\mathcal{T}_2)$ is said to be  continuous with respect to $L$-fuzzy topologies $\mathcal{T}_1$ and $\mathcal{T}_2$ if $\mathcal{T}_1(f_L^\leftarrow(V))\geq \mathcal{T}_2(V)$ holds for all $V\in L^Y$.\medskip

For  $a\in L$ and  the function $\mathcal{T}:L^{X}\rightarrow L$, we use the following notation from \cite{4}. \[\mathcal{T}_{[a]} =\{ A\in L^X|\mathcal{T}(A)\ge a\}.\]

\begin{thm}[\cite{Zhang}] Let $\mathcal{T}:L^X\rightarrow L$ be a function. Then the following conditions are equivalent:\begin{enumerate}\item[(1)] $\mathcal{T}$ is an $L$-fuzy topology on $X$;
\item[(2)] $\mathcal{T}_{[a]}$ is an $L$-topology on $X$, for each $a\in J(L)$.
\end{enumerate}
\end{thm}

\begin{thm}[\cite{Chattopadhyay,Kim}] Let $(X,\mathcal{T})$ be an $L$-fuzzy topological space. Then for each $r\in L_{\bot}$, $U\in L^{X}$ we define an operator $C_{\mathcal{T}}\colon L^{X}\times L_{\bot}\rightarrow L^{X}$ as
follows: \[
C_{\mathcal{T}}(U,r)=\bigwedge\{V\in L^{X}\mid
U\leq V,\, \mathcal{T}(V')\geq r\}.
\] For $U$, $V\in L^{X}$ and $r$, $s\in L_{\bot}$,
 the operator $C_{\mathcal{T}}$ satisfies the following statements:
\begin{enumerate}
\item[(C1)]$C_{\mathcal{T}}(\underline{\bot},r)=\underline{\bot}$.
\item[(C2)]$U\leq C_{\mathcal{T}}(U,r)$.
\item[(C3)]$C_{\mathcal{T}}(U,r)\vee C_{\mathcal{T}}(V,r)= C_{\mathcal{T}}(U\vee V,r).$
\item[(C4)]$C_{\mathcal{T}}(U,r)\leq C_{\mathcal{T}}(U,s)$ if $r\leq s$.
\item[(C5)]$C_{\mathcal{T}}(C_{\mathcal{T}}(U,r),r)=C_{\mathcal{T}}(U,r)$.
\end{enumerate}
\end{thm}

\begin{thm}[\cite{kim1}]Let $(X,\mathcal{T})$ be an $L$-fuzzy topological space. Then for each $r\in L_{\bot}$, $U\in L^{X}$ we define an operator $\mathcal{I}_{\mathcal{T}}\colon L^{X}\times L_{\bot}\rightarrow L^{X}$ as follows:
\[\mathcal{I}_{\mathcal{T}}(U,r)=\bigvee\{V\in L^{X}\mid
V\leq U,\, \mathcal{T}(V)\geq r\}.\]
For $U$, $V\in L^{X}$ and $r$, $s\in L_{\bot}$, the operator
$\mathcal{I}_{\mathcal{T}}$  satisfies the following statements:
\begin{enumerate}
\item[(I1)]$\mathcal{I}_{\mathcal{T}}(U',r)=C_{\mathcal{T}}(U,r)'$.
\item[(I2)]$\mathcal{I}_{\mathcal{T}}(\underline{\top},r)=\underline{\top}$.
\item[(I3)]$\mathcal{I}_{\mathcal{T}}(U,r)\leq U$.
\item[(I4)]$\mathcal{I}_{\mathcal{T}}(U,r)\wedge \mathcal{I}_{\mathcal{T}}(V,r)= \mathcal{I}_{\mathcal{T}}(U\wedge V,r).$
\item[(I5)]$\mathcal{I}_{\mathcal{T}}(U,s)\leq \mathcal{I}_{\mathcal{T}}(U,r)$ if $r\leq s$.
\item[(I6)]$\mathcal{I}_{\mathcal{T}}(\mathcal{I}_{\mathcal{T}}(U,r),r)=\mathcal{I}_{\mathcal{T}}(U,r)$.
\item[(I7)]If $\mathcal{I}_{\mathcal{T}}(C_{\mathcal{T}}(U,r),r)=U$, then
$C_{\mathcal{T}}(\mathcal{I}_{\mathcal{T}}(U',r),r)=U'$.
\end{enumerate}
\end{thm}

\begin{defn} Let $(X,\mathcal{T})$ be an $L$-fuzzy topological space. For $U\in L^X$ and $r\in L$, $U$ is called:\begin{enumerate}
\item[(1)]An $r$-fuzzy preopen (resp. $r$-fuzzy preclosed) \cite{kim1} if $U\leq \mathcal{I}_{\mathcal{T}}(C_{\mathcal{T}}(U,r),r)$ (resp. $C_{\mathcal{T}}(\mathcal{I}_{\mathcal{T}}(U,r),r)\leq U$).
\item[(2)]An $r$-fuzzy regular open (resp. $r$-fuzzy regular closed) \cite{2Lee} if $U=\mathcal{I}_{\mathcal{T}}(C_{\mathcal{T}}(U,r),r)$ (resp. $U=C_{\mathcal{T}}(\mathcal{I}_{\mathcal{T}}(U,r),r)$).
\item[(3)]An $r$-fuzzy $\alpha$-open (resp. $r$-fuzzy $\alpha$-closed)\cite{kim1} if $U\leq\mathcal{I}_{\mathcal{T}}(C_{\mathcal{T}}(\mathcal{I}_{\mathcal{T}}(U,r),r),r)$(resp. $C_{\mathcal{T}}(\mathcal{I}_{\mathcal{T}}(C_{\mathcal{T}}(U,r),r),r)\leq U$).
\end{enumerate}
\end{defn}

\begin{defn} Let $(X,\mathcal{T})$ be an $L$-fuzzy topological space. For $U\in L^X$ and $r\in L$, $U$ is called an $r$-fuzzy Semi-Preopen subset (resp. $r$-fuzzy Semi-Preclosed) if $U\leq C_{\mathcal{T}}(\mathcal{I}_{\mathcal{T}}(C_{\mathcal{T}}(U,r),r),r)$ (res. $\mathcal{I}_{\mathcal{T}}(C_{\mathcal{T}}(\mathcal{I}_{\mathcal{T}}(U,r),r),r)\leq U$).
\end{defn}
The Semi-Preinterior $sp\mathcal{I}_{\mathcal{T}}$ and Semi-Preclosur $spC_{\mathcal{T}}$ operators in $L$-fuzzy topological space $(X,\mathcal{T})$ defined as follows:\[sp\mathcal{I}_{\mathcal{T}}(U,r)=\bigvee\{V\in L^X|V\leq U\quad\mbox{and}\quad V\mbox{  is $r$-fuzzy Semi-Preopen}\},\]\[spC_{\mathcal{T}}(U,r)=\bigwedge\{V\in L^X|U\leq V\quad\mbox{and}\quad V\mbox{  is $r$-fuzzy Semi-Preclosed}\}.\]

\begin{defn}Let $f: (X,\mathcal{T}_{1})\rightarrow (Y,\mathcal{T}_{2})$ be a function from an $L$-fuzzy topological space $(X,\mathcal{T}_{1})$ into an $L$-fuzzy topological space $(Y,\mathcal{T}_{2})$. The function $f_L^\rightarrow$ is called:\begin{enumerate}
\item[(1)] An $L$-fuzzy Semi-Preclosed (resp. $L$-fuzzy Semi-Preopen)if $f_L^\rightarrow(U)$ is $r$-fuzzy Semi-Preclosed (resp. $r$-fuzzy Semi-Preopen )subset in $L^{Y}$ for each $U\in L^{X}$ and $r\in L$ such that $\mathcal{T}_{1}(U')\geq r$ (resp. $\mathcal{T}_{1}(U)\geq r$),
\item[(2)] An $L$-fuzzy almost open \cite{kim1}, if $\mathcal{T}_{2}(f_L^\rightarrow(U))\geq r$ for each
$r$-fuzzy regular open subset $U\in L^{X}$ and $r\in L$,
\item[(2)] An $L$-fuzzy strongly continuous \cite{kim1}, if $f_L^\rightarrow(C_{\mathcal{T}_{1}}(U,r))\leq f_L^\rightarrow(U)$ for every
$U\in L^{X}$ and $r\in L$,
\item[(3)] An $L$-fuzzy weakly open (resp. $L$-fuzzy weakly closed)\cite{kim1}, if $f_L^\rightarrow(U)\leq\mathcal{I}_{\mathcal{T}_{2}}(f_L^\rightarrow(C_{\mathcal{T}_{1}}(U,r)),r)$ (resp. $C_{\mathcal{T}_{2}}(f_L^\rightarrow(\mathcal{I}_{\mathcal{T}_{1}}(U,r)),r)\leq f_L^\rightarrow(U)$ ) for each $U\in L^{X}$ and $r\in L$ such that $\mathcal{T}_{1}(U)\geq r$ (resp. $\mathcal{T}_{1}(U')\geq r$).
\end{enumerate}
\end{defn}

\begin{defn}\cite{Dimirci} Let $(X,\mathcal{T})$ be an $L$-fuzzy topological space, $U\in L^X$, $x_\lambda\in J(L^X)$ and $r\in L$. $U$ is called an $r$-fuzzy Q-neighborhood of $x_\lambda$ if $\mathcal{T}(U)\geq r$ and $x_\lambda q U$.
\end{defn}
We will denote the set of all $r$-fuzzy open Q-neighborhood of $x_\lambda$ by $\mathcal{Q}_{\mathcal{T}}(x_\lambda,r)$.

\begin{defn}\cite{kim2} Let $(X,\mathcal{T})$ be an $L$-fuzzy topological space, $U\in L^X$, $x_\lambda\in J(L^X)$ and $r\in L$. $x_\lambda$ is called an $r$-fuzzy $\theta$-cluster point of $U$ if for each $V\in\mathcal{Q}_{\mathcal{T}}(x_\lambda,r)$, we have $C_{\mathcal{T}}(V,r)q U$.
\end{defn}
We denote $D_{\mathcal{T}}(U,r)=\bigvee\{x_{\lambda}\in J(L^X)|
 x_{\lambda}\mbox{   is $r$-fuzzy $\theta$-cluster point
 of $U$}\}$. Where $D_{\mathcal{T}}(U,r)$ is called
 $r$-fuzzy $\theta$-closure of $U$.

\begin{thm}[\cite{kim2}] Let $(X,\mathcal{T})$ an $L$-fuzzy topological space. For $U$, $V\in L^{X}$ and $r$, $s\in L$, we have the following:\begin{enumerate}
\item[(1)] $D_{\mathcal{T}}(U,r)=\bigwedge\{V\in L^{X}| U\leq \mathcal{I}_{\mathcal{T}}(V,r),\,\mathcal{T}(V')\geq r\}$.
\item[(2)] $x_{\lambda}$ is $r$-fuzzy $\theta$-cluster point of $U$ iff $x_{\lambda}\in D_{\mathcal{T}}(U,r)$.
\item[(3)] $C_{\mathcal{T}}(U,r)\leq D_{\mathcal{T}}(U,r)$.
\item[(4)] If $\mathcal{T}(U)\geq r$, then $C_{\mathcal{T}}(U,r)=D_{\mathcal{T}}(U,r)$.
\item[(5)]If $U$ is $r$-fuzzy preopen, then $C_{\mathcal{T}}(U,r)=D_{\mathcal{T}}(U,r)$.
\item[(6)] If $U$ is $r$-fuzzy preopen and $\lambda=C_{\mathcal{T}}(\mathcal{I}_{\mathcal{T}}(U,r),r)$, then $D_{\mathcal{T}}(U,r)=U$.
\end{enumerate}
\end{thm}
The complement of $r$-fuzzy $\theta$-closed set is called $r$-fuzzy
$\theta$-open and the $r$-fuzzy $\theta$-interior operator denoted
by $T_{\mathcal{T}}(U,r)$ is defined by
\[T_{\mathcal{T}}(U,r)=\bigvee\{V\in L^{X}| C_{\mathcal{T}}(V,r)\leq
U,\,\mathcal{T}(V)\geq r\}.\]

\section{$L$-Fuzzy Weakly Semi-Preopen ( Semi-Preclosed ) functions}

\begin{defn} A function $f:(X,\mathcal{T}_1)\rightarrow (Y,\mathcal{T}_2)$ is said to be:\begin{enumerate} \item[(a)] An $L$-fuzzy weakly Semi-Preopen function if $f_L^\rightarrow(U)\leq  sp\mathcal{I}_{\mathcal{T}_2}(f_L^\rightarrow(C_{\mathcal{T}_1}(U,r)),r)$ for each $U\in L^X$ and $r\in L$ such that $\mathcal{T}_1(U)\geq r$.
\item[(b)] An $L$-fuzzy weakly Semi-Preclosed function if $spC_{\mathcal{T}_2}(f_L^\rightarrow(\mathcal{I}_{\mathcal{T}_1}(U,r)),r)\leq f_L^\rightarrow(U)$ for each $U\in L^X$ and $r\in L$ such that $\mathcal{T}_1(U')\geq r$.\end{enumerate}
\end{defn}

\begin{rmk}\begin{enumerate}\item Every $L$-fuzzy weakly open function is $L$-fuzzy weakly Semi-Preopen function and every $L$-fuzzy Semi-preopen function is also $L$-fuzzy weakly Semi-preopen function. But the converse need not be true in general.
\item Every $L$-fuzzy weakly closed function is $L$-fuzzy weakly Semi-Preclosed but the converse need not be true in general.
\end{enumerate}
\end{rmk}

\begin{cexa}\begin{enumerate}
\item[(1)] Let $L=[0,1]$, $X=\{a,b,c\}$ and $Y=\{x,y,z\}$. The fuzzy subsets $U$, $V$ and $W$ are defined as: \begin{eqnarray*} U(a)=0.5,\quad &U(b)=0.3&,\quad U(c)=0.2;\\
    V(x)=0.9,\quad &V(y)=1&,\quad V(z)=0.7;\\
    W(x)=0.2,\quad &W(y)=0.9&,\quad W(z)=0.3.
    \end{eqnarray*}
    Let $\mathcal{T}_1:I^X\rightarrow I$ and $\mathcal{T}_2:I^Y\rightarrow I$ defined as follows:
\[ \mathcal{T}_1(A)=\left\{\begin{array}{ll}
1, & \hbox{if $A=\underline{0}$ or $\underline{1}$;} \\
\frac{1}{2}, & \hbox{if $A=U$;} \\
0, & \hbox{otherwise.}
\end{array}\right.\quad\mbox{and}\quad \mathcal{T}_2(A)=\left\{\begin{array}{ll}
1, & \hbox{if $B=\underline{0}$ or $\underline{1}$;} \\
\frac{1}{2}, & \hbox{if $B=V$;} \\
\frac{1}{4}, & \hbox{if $B=W$;} \\
0, & \hbox{otherwise.}
\end{array}\right.\] Then the function $f:(X,\mathcal{T}_1)\rightarrow (Y,\mathcal{T}_2)$ defined by \[f(a)=z,\quad f(b)=x,\quad f(c)=y,\] is $I$-fuzzy weakly Semi-Preopen but not $I$-fuzzy weakly open.
\item[(2)] Let $L=I=[0,1]$ and $X=\{a,b,c\}$. The fuzzy subsets $U$, $V$, $W$ and $H$ are defined as follows:\begin{eqnarray*} U(a)=0.4,\quad &U(b)=0.7&,\quad A(c)=0.2;\\
V(a)=0.3,\quad &V(b)=0.1&,\quad V(c)=0.6;\\
W(a)=0.5,\quad &W(b)=0.8&,\quad W(c)=0.3;\\
H(a)=0.4,\quad &H(b)=0.2&,\quad H(c)=0.7.
\end{eqnarray*}
Let $\mathcal{T}_1$, $\mathcal{T}_2:I^X\rightarrow I$ defined as follows:

\[ \mathcal{T}_1(A)=\left\{\begin{array}{ll}
1, & \hbox{if $A=\underline{0}$ or $\underline{1}$;} \\
\frac{1}{2}, & \hbox{if $A=U$;} \\
\frac{1}{4}, & \hbox{if $A=V$;} \\
0, & \hbox{otherwise.}
\end{array}\right.\quad\mbox{and}\quad \mathcal{T}_2(A)=\left\{\begin{array}{ll}
1, & \hbox{if $A=\underline{0}$ or $\underline{1}$;} \\
\frac{1}{2}, & \hbox{if $A=W$;} \\
\frac{1}{4}, & \hbox{if $A=H$;} \\
0, & \hbox{otherwise.}
\end{array}\right.\]
Then the identity function $i:(X,\mathcal{T}_1)\rightarrow (X,\mathcal{T}_2)$ is $I$-fuzzy weakly Semi-Preopen function but not $I$-fuzzy Semi-Preopen.
\item[(3)] Let $L=[0,1]$, $X=\{a,b\}$ and $Y=\{x,y\}$. The fuzzy subsets $U$, $V$ are defined as follows:\begin{eqnarray*}U(a)=0.5&,&\quad U(b)=0.6;\\
     V(x)=0.4&,&\quad V(y)=0.3.
    \end{eqnarray*}
    Let $\mathcal{T}_1:I^X\rightarrow I$ and $\mathcal{T}_2:I^Y\rightarrow I$ defined as follows:
\[ \mathcal{T}_1(A)=\left\{\begin{array}{ll}
1, & \hbox{if $A=\underline{0}$ or $\underline{1}$;} \\
\frac{1}{2}, & \hbox{if $A=U$;} \\
0, & \hbox{otherwise.}
\end{array}\right.\quad\mbox{and}\quad \mathcal{T}_2(B)=\left\{\begin{array}{ll}
1, & \hbox{if $B=\underline{0}$ or $\underline{1}$;} \\
\frac{1}{2}, & \hbox{if $B=V$;} \\
0, & \hbox{otherwise.}
\end{array}\right.\] The function $f:(X,\mathcal{T}_1)\rightarrow (Y,\mathcal{T}_2)$ defined by\[ f(a)=x,\quad f(b)=y,\] is $I$-fuzzy weakly Semi-Preclosed but not $I$-fuzzy weakly closed.
\end{enumerate}
\end{cexa}

\begin{thm}Let $(X,\mathcal{T}_1)$ and $(Y,\mathcal{T}_2)$ are $L$-fuzzy topological spaces. The function $f:(X,\mathcal{T}_1)\rightarrow (Y,\mathcal{T}_2)$ is $L$-fuzzy weakly Semi-Preopen (resp. $L$-fuzzy weakly Semi-Preclosed) iff $f:(X,{\mathcal{T}_1}_{[a]})\rightarrow (Y,{\mathcal{T}_2}_{[a]})$ for each $a\in J(L)$.
\end{thm}

\begin{proof}Straightforward.
\end{proof}

\begin{thm}For a function $f:(X,\mathcal{T}_1)\rightarrow (Y,\mathcal{T}_2)$, the following conditions are equivalent:\begin{enumerate}
\item[(1)] $f_L^\rightarrow$ is $L$-fuzzy weakly Semi-Preopen function;
\item[(2)] $f_L^\rightarrow(T_{\mathcal{T}_1}(U,r))\leq sp\mathcal{I}_{\mathcal{T}_2}(f_L^\rightarrow(U),r)$ for each $U\in L^X$ and $r\in L$;
\item[(3)] $T_{\mathcal{T}_1}(f_L^\leftarrow(V),r)\leq f_L^\leftarrow(sp\mathcal{I}_{\mathcal{T}_2}(V,r))$ for each $V\in L^Y$ and $r\in L$;
\item[(4)] $f_L^\leftarrow(spC_{\mathcal{T}_2}(V,r))\leq D_{\mathcal{T}_1}(f_L^\leftarrow(V),r)$ for each $V\in L^Y$ and $r\in L$;
\item[(5)] For each $x_\lambda\in J(L^X)$ and $U\in L^X$ such that $\mathcal{T}_1(\lambda)\geq r$ and $x_\lambda\leq U$ there exists an $r$-fuzzy Semi-Preopen subset $V$ such that $f_L^\rightarrow(x_\lambda)\leq V$ and $V\leq f_L^\rightarrow(C_{\mathcal{T}_1}(U,r))$.
\end{enumerate}
\end{thm}

\begin{proof}
\begin{description}
  \item[(1)$\Rightarrow$(2):] Let $x_\lambda\leq T_{\mathcal{T}_1}(U,r)$. Then $f_L^\rightarrow (x_\lambda)\leq f_L^\rightarrow(U)$. Since $f_L^\rightarrow$ is $L$-fuzzy weakly Semi-Preopen function, then\[f_L^\rightarrow (x_\lambda)\leq f_L^\rightarrow(U)\leq sp\mathcal{I}_{\mathcal{T}_2}(f_L^\rightarrow(C_{\mathcal{T}_1}(U,r)),r).\] Therefore, $x_\lambda\leq f_L^\leftarrow(sp\mathcal{I}_{\mathcal{T}_2}(f_L^\rightarrow(C_{\mathcal{T}_1}(U,r)),r))$. Thus $T_{\mathcal{T}_1}(U,r)\leq f_L^\leftarrow(sp\mathcal{I}_{\mathcal{T}_2}(f_L^\rightarrow(C_{\mathcal{T}_1}(U,r)),r))$, i.e.,  \[f_L^\rightarrow(T_{\mathcal{T}_1}(U,r))\leq sp\mathcal{I}_{\mathcal{T}_2}(f_L^\rightarrow(C_{\mathcal{T}_1}(U,r)),r).\]
  \item[(2)$\Rightarrow$(1):] Let $U\in L^X$ and $r\in L$ such that $\mathcal{T}_1(U)\geq r$. Since $U\leq T_{\mathcal{T}_1}(C_{\mathcal{T}_1}(U,r),r)$ and by using (2), we have\[f_L^\rightarrow(U)\leq f_L^\rightarrow(T_{\mathcal{T}_1}(C_{\mathcal{T}_1}(U,r),r))\leq sp\mathcal{I}_{\mathcal{T}_2}(f_L^\rightarrow(C_{\mathcal{T}_1}(U,r)),r).\] Hence $f_L^\rightarrow$ is $L$-fuzzy weakly Semi-Preopen function.
  \item[(2)$\Rightarrow$(3):] Let $V\in L^Y$. By using (2), we have $f_L^\rightarrow(T_{\mathcal{T}_1}(f_L^\leftarrow(V),r))\leq sp\mathcal{I}_{\mathcal{T}_2}(V,r)$. Therefore, $T_{\mathcal{T}_1}(f_L^\leftarrow(V),r)\leq f_L^\leftarrow(sp\mathcal{I}_{\mathcal{T}_2}(V,r)).$
\item[(3)$\Rightarrow$(2):] It is trivial and omitted.\medskip

\item[(3)$\Rightarrow$(4):] Let $V\in L^Y$ and $r\in L$. By using (3), we have \begin{eqnarray*}D_{\mathcal{T}_1}(f_L^\leftarrow(V),r)'&=&T_{\mathcal{T}_1}(f_L^\leftarrow(V)',r)=T_{\mathcal{T}_1}(f_L^\leftarrow(V'),r)\leq f_L^\leftarrow (sp\mathcal{I}_{\mathcal{T}_2}(V',r))\\
    &=& f_L^\leftarrow(spC_{\mathcal{T}_2}(V,r)')=(f_L^\leftarrow (spC_{\mathcal{T}_2}(V,r)))'.\end{eqnarray*}
    Therefore, we obtain $f_L^\leftarrow(spC_{\mathcal{T}_2}(V,r))\leq D_{\mathcal{T}_1}(f_L^\leftarrow(V),r)$.
\item[(4)$\Rightarrow$(3):] It is trivial and omitted.\medskip
\item[(1)$\Rightarrow$(5):] Let $x_\lambda\in J(L^X)$ and $U\in L^X$ such that $\mathcal{T}_1(U)\geq r$ and $x_\lambda\leq U$. Since $f_L^\rightarrow$ is $L$-fuzzy weakly Semi-Preopen, $f_L^\rightarrow(U)\leq sp\mathcal{I}_{\mathcal{T}_2}(f_L^\rightarrow(C_{\mathcal{T}_1}(U,r)),r)$. Let $V=sp\mathcal{I}_{\mathcal{T}_2}(f_L^\rightarrow(C_{\mathcal{T}_1}(U,r)),r)$. Then $V\leq f_L^\rightarrow(C_{\mathcal{T}_1}(U,r))$ with $f_L^\rightarrow(x_\lambda)\leq V$.
\item[(5)$\Rightarrow$(1):] Let $U\in L^X$ and $r\in L$ such that $\mathcal{T}_1(U)\geq r$ and let $y_\beta\leq f_L^\rightarrow(U)$. By using (2), we have $V\leq f_L^\rightarrow(C_{\mathcal{T}_2}(U,r))$ for some $r$-fuzzy Semi-Preopen subset $V\in L^Y$ and $y_\beta\leq V$. Hence we have, $y_\beta\leq V\leq sp\mathcal{I}_{\mathcal{T}_2}(f_L^\rightarrow(C_{\mathcal{T}_1}(U,r)),r)$. This shows that $f_L^\rightarrow (U)\leq sp\mathcal{I}_{\mathcal{T}_2}(f_L^\rightarrow(C_{\mathcal{T}_1}(U,r)),r)$. \end{description}
\end{proof}

\begin{thm} For a function $f:(X,\mathcal{T}_1)\rightarrow (Y,\mathcal{T}_2)$, the following conditions are equivalent:\begin{enumerate}
\item[(1)] $f_L^\rightarrow$ is $L$-fuzzy weakly Semi-Preopen function;
\item[(2)] $f_L^\rightarrow(\mathcal{I}_{\mathcal{T}_1}(U,r))\leq sp\mathcal{I}_{\mathcal{T}_2}(f_l^\rightarrow(U),r)$ for each $U\in L^X$ and $r\in L$ such that $\mathcal{T}_1(U')\geq r$;
\item[(3)] $f_L^\rightarrow(\mathcal{I}_{\mathcal{T}_1}(C_{\mathcal{T}_1}(U,r),r))\leq sp\mathcal{I}_{\mathcal{T}_2}(f_L^\rightarrow(C_{\mathcal{T}_1}(U,r)),r)$ for each $U\in L^X$ and $r\in L$ such that $\mathcal{T}_1(U)\geq r$;
\item[(4)] $f_L^\rightarrow(U)\leq sp\mathcal{I}(f_L^\rightarrow(C_{\mathcal{T}_1}(U,r)),r)$ for each $r$-fuzzy Preopen subset $U\in L^X$;
\item[(5)] $f_L^\rightarrow(U)\leq sp\mathcal{I}_{\mathcal{T}_2}(f_L^\rightarrow(C_{\mathcal{T}_1}(U,r)),r)$ for each $r$-fuzzy $\alpha$-open subset $U\in L^X$.
\end{enumerate}
\end{thm}

\begin{proof}
\begin{description}
\item[(1)$\Rightarrow$(2):]  Let $\mathcal{T}_1(U')\geq r$ for each $U\in L^X$ and $r\in L$, then $\mathcal{I}_{\mathcal{T}_1}(U,r)=\mathcal{I}_{\mathcal{T}_1}(C_{\mathcal{T}_1}(U,r),r)$. By using (1), \begin{eqnarray*}f_L^\rightarrow(\mathcal{I}_{\mathcal{T}_1}(U,r))&=&f_L^\rightarrow(\mathcal{I}_{\mathcal{T}_1}(C_{\mathcal{T}_1}(U,r),r))\leq sp\mathcal{I}_{\mathcal{T}_2}(f_L^\rightarrow(\mathcal{I}_{\mathcal{T}_1}(C_{\mathcal{T}_1}(U,r),r))\\
&=&sp\mathcal{I}_{\mathcal{T}_2}(f_L^\rightarrow(C_{\mathcal{T}_1}(U,r)),r).\end{eqnarray*}
for each $U\in L^X$ and $r\in L$ such that $\mathcal{T}_1(U')\geq r$.\medskip

\item[(2)$\Rightarrow$(3):] It is trivial and omitted.

\item[(3)$\Rightarrow$(4):] Let $U\in L^X$ be an $r$-fuzzy Preopen subset, then $U\leq \mathcal{I}_{\mathcal{T}_1}(C_{\mathcal{T}_1}(U,r),r)$. By using (3), we have \[f_L^\rightarrow(U)\leq f_L^\rightarrow(\mathcal{I}_{\mathcal{T}_1}(C_{\mathcal{T}_1}(U,r),r))\leq sp\mathcal{I}_{\mathcal{T}_2}(f_L^\rightarrow(C_{\mathcal{T}_1}(U,r)),r).\]

\item[(4)$\Rightarrow$(5):] It is trivial and omitted.

\item[(5)$\Rightarrow$(1):] It is trivial and omitted.
\end{description}
\end{proof}

\begin{thm}\label{thm3}If $f:(X,\mathcal{T}_1)\rightarrow (Y,\mathcal{T}_2)$ is $L$-fuzzy weakly Semi-Preopen and $L$-fuzzy strongly continuous, then $f_L^\rightarrow$ is $L$-fuzzy Semi-Preopen.
\end{thm}

\begin{proof}Let $U\in L^X$ and $r\in L$ such that $\mathcal{T}_1(U)\geq r$. Since $f_L^\rightarrow$ is $L$-fuzzy weakly Semi-Preopen \[f_L^\rightarrow(U)\leq sp\mathcal{I}_{\mathcal{T}_2}(f_L^\rightarrow(C_{\mathcal{T}_1}(U,r)),r).\] However, since $f_L^\rightarrow$ is $L$-fuzzy strongly continuous, $f_L^\rightarrow(U)\leq sp\mathcal{I}_{\mathcal{T}_2}(f_L^\rightarrow(U),r)$ and therefore $f_L^\rightarrow(U)$ is $r$-fuzzy Semi-Preopen subset.
\end{proof}

\begin{thm}If $f:(X,\mathcal{T}_1)\rightarrow (Y,\mathcal{T}_2)$ is $L$-fuzzy almost open function, then $f_L^\rightarrow$ is $L$-fuzzy weakly Semi-Preopen.
\end{thm}

\begin{proof}Let $U\in L^{X}$ and $r\in L$ such that $\mathcal{T}_{1}(U)\geq r$. Since $f_L^\rightarrow$ is $L$-fuzzy almost open and $\mathcal{I}_{\mathcal{T}_{1}}(C_{\mathcal{T}_{1}}(U,r),r)$ is $r$-fuzzy regular open,
then \[\mathcal{I}_{\mathcal{T}_{2}}(f_L^\rightarrow(\mathcal{I}_{\mathcal{T}_{1}}(C_{\mathcal{T}_{2}}(U,r),r)),r)=f_L^\rightarrow(\mathcal{I}_{\mathcal{T}_{1}}(C_{\mathcal{T}_{2}}(U,r),r))\] and
hence\begin{eqnarray*} f_L^\rightarrow(\lambda)&\leq & f_L^\rightarrow(\mathcal{I}_{\mathcal{T}_{1}}(C_{\mathcal{T}_{1}}(U,r),r)\leq \mathcal{I}_{\mathcal{T}_{2}}(f_L^\rightarrow(C_{\mathcal{T}_{1}}(U,r)),r)\\
&\leq & sp\mathcal{I}_{\mathcal{T}_{2}}(f_L^\rightarrow(C_{\mathcal{T}_{1}}(U,r)),r).
\end{eqnarray*}This shows that $f_L^\rightarrow$ is $L$-fuzzy weakly Semi-Preopen.
\end{proof}

\begin{thm}\label{thm1}For the function $f:(X,\mathcal{T}_1)\rightarrow (Y,\mathcal{T}_2)$, the following conditions are equivalent:
\begin{enumerate}
\item[(1)] $f_L^\rightarrow$ is $L$-fuzzy weakly Semi-Preclosed;
\item[(2)] $spC_{\mathcal{T}_2}(f_L^\rightarrow(U),r)\leq f_L^\rightarrow (C_{\mathcal{T}_1}(U,r))$ for each $U\in L^X$ and $r\in L$ such that $\mathcal{T}_1(U)\geq r$;
\item[(3)] $spC_{\mathcal{T}_2}(f_L^\rightarrow(U),r))\leq f_L^\rightarrow(C_{\mathcal{T}_1}(U,r))$ for each an $r$-fuzzy regular open subset $U\in L^X$;
\item[(4)] For each $V\in L^Y$, $U\in L^X$ and $r\in L$ such that $\mathcal{T}_1(U)\geq r$ and $f_L^\leftarrow(V)\leq U$, there exists an $r$-fuzzy Semi-Preopen subset $W\in L^Y$ with $V\leq W$ and $f_L^\leftarrow (V)\leq C_{\mathcal{T}_1}(U,r)$;
\item[(5)] For each $x_\lambda\in J(L^Y)$, $U\in L^X$ and $r\in L$ such that $f_L^\leftarrow(x_\lambda)\leq U$, there exists an $r$-fuzzy Semi-Preopen subset $V\in L^Y$ with $x_\lambda\leq V$ and $f_L^\leftarrow(V)\leq C_{\mathcal{T}_1}(U,r)$.
\item[(6)] $spC_{\mathcal{T}_{2}}(f_L^\rightarrow(\mathcal{I}_{\mathcal{T}_1}(C_{\mathcal{T}_1}(U,r),r)),r)\leq f_L^\rightarrow(C_{\mathcal{T}_1}(U,r))$ for each $U\in L^X$ and $r\in L$;
\item[(7)] $spC_{\mathcal{T}_2}(f_L^\rightarrow(\mathcal{I}_{\mathcal{T}_1}(D_{\mathcal{T}_1}(U,r),r)),r)\leq f_L^\rightarrow (D_{\mathcal{T}_1}(U,r))$ for each $U\in L^X$ and $r\in L$;
\item[(8)] $spC_{\mathcal{T}_2}(f_L^\rightarrow(U),r)\leq f_L^\rightarrow(C_{\mathcal{T}_1}(U,r))$ for each $r$-fuzzy Semi-Preopen subset $U\in L^X$.
\end{enumerate}
\end{thm}

\begin{proof}\begin{description}
\item[(1)$\Rightarrow$(2):] Let $U\in L^X$ and $r\in L$ such that $\mathcal{T}_1(U)\geq r$. Then \begin{eqnarray*}spC_{\mathcal{T}_2}(f_L^\rightarrow(U),r)&=& spC_{\mathcal{T}_2}(f_L^\rightarrow(\mathcal{I}_{\mathcal{T}_1}(U,r)))\leq spC_{\mathcal{T}_2}(f_L^\rightarrow(\mathcal{I}_{\mathcal{T}_1}(C_{\mathcal{T}_1}(U,r),r)),r)\\
    &\leq &f_L^\rightarrow(C_{\mathcal{T}_1}(U,r)).\end{eqnarray*}
\item[(2)$\Rightarrow$(1):] Let $U\in L^X$ and $r\in L$ such that $\mathcal{T}(U')\geq r$. Then\begin{eqnarray*}spC_{\mathcal{T}_2}(f_L^\rightarrow(\mathcal{I}_{\mathcal{T}_1}(U,r)),r)&\leq & f_L^\rightarrow(C_{\mathcal{T}_1}(\mathcal{I}_{\mathcal{T}_1}(U,r),r))\leq f_L^\rightarrow (C_{\mathcal{T}_1}(U,r))\\
    &=& f_L^\rightarrow(U).\end{eqnarray*}

\item[(3)$\Rightarrow$(4):] Let $U\in L^X$, $V\in L^Y$ and $r\in L$ such that $\mathcal{T}_1(U)\geq r$ and $f_L^\leftarrow(V)\leq U$. Then $f_L^\leftarrow(V)\neg q C_{\mathcal{T}_1}(C_{\mathcal{T}_1}(U,r)',r)$. This implies to $V\neg q f_L^\rightarrow(C_{\mathcal{T}_1}(C_{\mathcal{T}_1}(U,r)',r))$. Since $C_{\mathcal{T}_1}(U,r)'$ is $r$-fuzzy regular open subset, $V\neg q spC_{\mathcal{T}_2}(f_L^\rightarrow(C_{\mathcal{T}_1}(U,r)'),r)$. Let $W=spC_{\mathcal{T}_2}(f_L^\rightarrow(C_{\mathcal{T}_1}(U,r)'),r)$. Then $W$ is $r$-fuzzy Semi-Preopen subset with $V\leq W$ and \[f_L^\leftarrow(W)= f_L^\leftarrow(spC_{\mathcal{T}_2}(C_{\mathcal{T}_1}(U,r)',r))'\leq f_L^\leftarrow(f_L^\rightarrow(C_{\mathcal{T}_1}(U,r)')')\leq C_{\mathcal{T}_1}(U,r).\]
\item[(5)$\Rightarrow$(1):] Let $V\in L^Y$ and $r\in L$ such that $\mathcal{T}_2(V')\geq r$ and $y_\beta\leq f_L^\rightarrow(V)'$. Since $f_L^\leftarrow(y_\beta)\leq V'$, there exists $r$-fuzzy Semi-Preopen subset $W\in L^Y$ such that $y_\beta\leq W$ and $f_L^\leftarrow(W)\leq C_{\mathcal{T}_1}(V',r)=\mathcal{I}_{\mathcal{T}_1}(V,r)'$. Therefore $W\neg q f_L^\rightarrow (\mathcal{I}_{\mathcal{T}_1}(V,r))$. Then $y_\beta\leq spC_{\mathcal{T}_2}(f_L^\rightarrow(\mathcal{I}_{\mathcal{T}_1}(V,r)),r)'$.
    \end{description}
\end{proof}

\begin{thm}
If the function $f:(X,\mathcal{T}_1)\rightarrow (Y,\mathcal{T}_2)$ is a bijective function. Then the following conditions are equivalent:\begin{enumerate}
\item[(1)] $f_L^\rightarrow$ is $L$-fuzzy weakly Semi-Preopen function;
\item[(2)] $spC_{\mathcal{T}_2}(f_L^\rightarrow(U),r)\leq f_L^\rightarrow(C_{\mathcal{T}_1}(U,r))$ for each $U\in L^X$ and $r\in L$ such that $\mathcal{T}_1(U)\geq r$;
\item[(3)] $spC_{\mathcal{T}_2}(f_L^\rightarrow(\mathcal{I}_{\mathcal{T}_1}(V,r)),r)\leq f_L^\rightarrow(V)$ for each $V\in L^X$ and $r\in L$ such that $\mathcal{T}_1(V')\geq r$.
\end{enumerate}
\end{thm}

\begin{proof}\begin{description}\item[(1)$\Rightarrow$(3):]  Let $U\in L^X$ and $r\in L$ such that $\mathcal{T}_1(U')\geq r$. Then \[(f_L^\rightarrow(U))'=f_L^\rightarrow(U')\leq sp\mathcal{I}_{\mathcal{T}_2}(f_L^\rightarrow(C_{\mathcal{T}_1}(U',r)),r)\]  and so \[(f_L^\rightarrow(U))'\leq (spC_{\mathcal{T}_2}(f_L^\rightarrow(\mathcal{I}_{\mathcal{T}_1}(U,r)),r))'.\] Hence, \[spC_{\mathcal{T}_2}(f_L^\rightarrow(\mathcal{I}_{\mathcal{T}_1}(U,r)),r)\leq f_L^\rightarrow(U).\]

\item[(3)$\Rightarrow$(2):] Let $U\in L^X$ and $r\in L$ such that $\mathcal{T}_1(U)\geq r$. Since $\mathcal{T}_{1}(C_{\mathcal{T}_1}(U,r)')\geq r$  and $U\leq\mathcal{I}_{\mathcal{T}_1}(C_{\mathcal{T}_1}(U,r),r)$ and by using (3), we have \[spC_{\mathcal{T}_2}(f_L^\rightarrow(U),r)\leq spC_{\mathcal{T}_{2}}(f_L^\rightarrow(\mathcal{I}_{\mathcal{T}_1}(C_{\mathcal{T}_1}(U,r),r)),r)\leq f_L^\rightarrow(C_{\mathcal{T}_1}(U,r)).\]

\item[(2)$\Rightarrow$(3):] It is trivial and omitted.

\item[(3)$\Rightarrow$(1):] It is trivial and omitted.
\end{description}
\end{proof}

\begin{thm}For the function $f:(X,\mathcal{T}_1)\rightarrow (Y,\mathcal{T}_2)$, the following conditions are equivalent:
\begin{enumerate}
\item[(1)] $f_L^\rightarrow$ is $L$-fuzzy weakly Semi-Preclosed;
\item[(2)] $spC_{\mathcal{T}_2}(f_L^\rightarrow(\mathcal{I}_{\mathcal{T}_1}(U,r)),r)\leq f_L^\rightarrow(U)$ for $r$-fuzzy Semi-Preclosed subset $U\in L^X$;
\item[(3)] $spC_{\mathcal{T}_2}(f_L^\rightarrow(\mathcal{I}_{\mathcal{T}_1}(U,r)),r)\leq f_L^\rightarrow (U)$ for each $r$-fuzzy $\alpha$-closed subset $U\in L^X$.
\end{enumerate}
\end{thm}

\begin{proof}Straightforward.
\end{proof}

\begin{thm}If  the function $f:(X,\mathcal{T}_1)\rightarrow (Y,\mathcal{T}_2)$ is $L$-fuzzy weakly Semi-Preopen and $L$-fuzzy strongly continuous, then $f_L^\rightarrow$ is $L$-fuzzy Semi-Preopen function.
\end{thm}

\begin{proof} Let $U\in L^X$ and $r\in L$ such that $\mathcal{T}_1(U)\geq r$.  Since $f_L^\rightarrow$ is $L$-fuzzy strongly contiinuous, $f_L^\rightarrow(C_{\mathcal{T}_1}(U,r))\leq f_L^\rightarrow(U)$ for each $U\in L^X$ and $r\in L$. But $f_L^\rightarrow$ is $L$-fuzzy weakly Semi-Preopen, then \[f_L^\rightarrow(U)\leq sp\mathcal{I}_{\mathcal{T}_2}(f_L^\rightarrow(C_{\mathcal{T}_1}(U,r)),r)\leq f_L^\rightarrow (U).\]This is implies to that $f_L^\rightarrow(U)$ is an $r$-fuzzy Semi-Preopen subset. Therefore, $f_L^\rightarrow$ is $L$-fuzzy Semi-Preopen function.
\end{proof}

\begin{defn} The function $f:(X,\mathcal{T}_1)\rightarrow (Y,\mathcal{T}_2)$ is said to be:
\begin{enumerate}
\item[(1)] an $L$-fuzzy contra-Semi-Preclosed function if $f_L^\rightarrow(U)$ is an $r$-fuzzy Semi-Preopen subset, for each $U\in L^X$ and $r\in L$ such that $\mathcal{T}_1(U')\geq r$.
\item[(2)] an $L$-fuzzy contra-Semi-Preopen if an $r$-fuzzy Semi-Preclosed subset, for each $U\in L^X$ and $r\in L$ such that $\mathcal{T}_1(U)\geq r$.
\end{enumerate}
\end{defn}

\begin{thm}If the function $f:(X,\mathcal{T}_1)\rightarrow (Y,\mathcal{T}_2)$ is $L$-fuzzy  contra-Semi-Preclosed (resp. $L$-fuzzy contra-Semi-Preopen), then it is an $L$-fuzzy weakly Semi-Preopen (resp. $L$-fuzzy weakly Semi-Preclosed).
\end{thm}

\begin{proof} Let $U\in L^X$ and $r\in L$ such that $\mathcal{T}_1(U)\geq r$ (resp. $\mathcal{T}_1(U')\geq r$). Then, we have $ f_L^\rightarrow(U)\leq f_L^\rightarrow(C_{\mathcal{T}_1}(U,r))=sp\mathcal{I}_{\mathcal{T}_2}(f_L^\rightarrow(C_{\mathcal{T}_1}(U,r)),r)$ (resp. $spC_{\mathcal{T}_2}(f_L^\rightarrow(\mathcal{I}_{\mathcal{T}_1}(U,r)),r)= f_L^\rightarrow(\mathcal{I}_{\mathcal{T}_1}(U,r))\leq f_L^\rightarrow(U)$).
\end{proof}

\section{Applications}
In this section, we will provide some applications for this kind of functions in some other $L$-fuzzy topological concepts.  We know that there are many other applications of $L$-fuzzy weakly Semi-Preopen and $L$-fuzzy weakly Semi-Preclosed functions, but we confine ourselves to the application in separation and connectedness.
\subsection{Separation:}

\begin{defn} Let $(X,\mathcal{T})$ be an $L$-fuzzy topological space. An $L$-fuzzy
subsets $U$, $V\in L^{X}$ are $r$-fuzzy strongly separated if
there exist $H$, $N\in L^{X}$ such that $\mathcal{T}(H)\geq r$,
$\mathcal{T}(N)\geq r$ with $U\leq H$, $V\leq N$ and $C_{\mathcal{T}}(H,r)\neg{q} C_{\mathcal{T}}(N,r)$.
\end{defn}

\begin{defn} An $L$-fuzzy topological space
$(X,\mathcal{T})$ is called $r$-Semi pre $T_{2}$ if for each
$x_{\lambda_{1}}$, $x_{\lambda_{2}}$ with different supports there exist
$r$-fuzzy Semi-Preopen  sets $U$, $V\in L^{X}$ such that
$x_{\lambda_{1}}\leq U\leq x_{\lambda_{2}}'$, $x_{\lambda_{2}}\leq V\leq
x_{\lambda_{1}}'$ and $U\neg{q}V$.
\end{defn}

\begin{thm}If $f\colon (X,\mathcal{T}_{1})\rightarrow
(Y,\mathcal{T}_{2})$ is $L$-fuzzy weakly Semi-Preclosed surjective function and all fibers are $r$-fuzzy strongly separated, then $(Y,\mathcal{T}_{1})$ is
$r$-Semi Pre-$T_{2}$.
\end{thm}

\begin{proof} Let $y_{\beta_{1}}$, $y_{\beta_{2}}\in
J(L^Y)$ and let $U$, $V\in L^{X}$, $r\in L$ such that
$\mathcal{T}_{1}(U)\geq r$, $\mathcal{T}_{1}(V)\geq r$,
$f_L^{\leftarrow}(y_{\beta_{1}})\leq U$ and $f_L^{\leftarrow}(y_{\beta_{2}})\leq V$
respectively with $C_{\mathcal{T}_{1}}(U,r)\neg{q}C_{\mathcal{T}_{1}}(V,r)$. By using Theorem \ref{thm1}, there are $r$-fuzzy Semi-Preopen  sets $H$, $N\in L^{Y}$ such that
$y_{\beta_{1}}\leq H$, $y_{\beta_{2}}\leq N$, $f_L^{\leftarrow}(U)\leq
C_{\mathcal{T}_{1}}(U,r)$ and $f_L^{\leftarrow}(N)\leq C_{\mathcal{T}_{1}}(V,r)$.
Therefore $H\neg{q}N$, because $C_{\mathcal{T}_{1}}(U,r)\neg{q}
C_{\mathcal{T}_{2}}(V,r)$ and $f_L^\rightarrow$ is surjective. Thus $(Y,\mathcal{T}_{2})$ is
$r$-fuzzy Semi-Pre-$T_{2}$.
\end{proof}

\subsection{Connectedness:}

\begin{defn} [\cite{spabbas}] Let $(X,\mathcal{T})$ be an $L$-fuzzy topological space and $r\in L$. The two $L$-fuzzy subsets $U$, $V\in L^{X}$ are
said to be $r$-fuzzy separated iff $U\neg{q} C_{\mathcal{T}}(V,r)$
and $V\neg{q} C_{\mathcal{T}}(U,r)$. An $L$-fuzzy subset which cannot be
expressed as the union of two $r$-fuzzy separated subsets is said to be
$r$-fuzzy connected.\end{defn}

\begin{defn} Let $(X,\mathcal{T})$ an $L$-fuzzy topological space. For $L$-fuzzy
subsets $U$, $V\in L^{X}$ such that $U\not =\underline{\bot}$ and $V\not
=\underline{\bot}$,  are said to be $r$-fuzzy Semi-Preseparated  if
$U\neg{q} spC_{\mathcal{T}}(V,r)$ and $V\neg{q}
spC_{\mathcal{T}}(U,r)$  or equivalently if there exist two $r$-fuzzy
Semi-Preopen subsets $H$, $N$ such that
$U\leq H$, $V\leq N$, $U\neg{q} N$ and
$V\neg{q} H$. An $L$-fuzzy topological space which can not be expressed as the union of
two  $r$-fuzzy Semi-Preseparated subsets is said to be $r$-fuzzy Semi-Preconnected space.
\end{defn}

\begin{thm}If $f:(X,\mathcal{T}_{1})\rightarrow (Y,\mathcal{T}_{2})$ is  injective $L$-fuzzy weakly Semi-Preopen  and $L$-fuzzy strongly continuous function of space $(X,\mathcal{T}_{1})$ onto an $r$-fuzzy Semi-Preconnected space $(Y,\mathcal{T}_{2})$, then $(X,\mathcal{T}_{1})$ is $r$-fuzzy connected.
\end{thm}

\begin{proof}Let $(X,\mathcal{T}_{1})$ be not $r$-fuzzy connected. Then there exist $r$-fuzzy separated sets $U$, $V\in L^{X}$ such that $U\vee V=\underline{\top}$. Since $U$ and $V$ are $r$-fuzzy separated, there exists $H$, $N\in L^{X}$ such that $\mathcal{T}_{1}(H)\geq r$,  $\mathcal{T}_{1}(N)\geq r$ such that $U\leq H$, $V\leq N$, $U \neg{q} N$, $V\neg{q}H$. Hence we have $f_L^\rightarrow(U)\leq f_L^\rightarrow(H)$, $f_L^\rightarrow(V)\leq f_L^\rightarrow(N)$, $f_L^\rightarrow(U)\neg q f_L^\rightarrow(N)$ and $f_L^\rightarrow(V)\neg q f_L^\rightarrow(H)$.  Since $f_L^\rightarrow$ is $L$-fuzzy weakly Semi-Preopen and $L$-fuzzy strongly continuous function, from Theorem \ref{thm3} we have  $f_L^\rightarrow(H)$ and $f_L^\rightarrow(N)$ are $r$-fuzzy Semi-Preopen. Therefore, $f_L^\rightarrow(U)$ and $f_L^\rightarrow(V)$ are $r$-fuzzy Semi Preseparated and \[\underline{\top}=f_L^\rightarrow(\underline{\top})=f_L^\rightarrow(U\vee V)=f_L^\rightarrow(U)\vee f_L^\rightarrow(V)\] which is contradiction with $(Y,\mathcal{T}_{2})$ is $r$-fuzzy Semi Preconnected. Thus $(X,\mathcal{T}_{1})$ is $r$-fuzzy connected.
\end{proof}

\providecommand{\bysame}{\leavevmode\hbox to3em{\hrulefill}\thinspace}
\providecommand{\MR}{\relax\ifhmode\unskip\space\fi MR }
\providecommand{\MRhref}[2]{%
  \href{http://www.ams.org/mathscinet-getitem?mr=#1}{#2}
}
\providecommand{\href}[2]{#2}

\end{document}